\documentclass[a4paper,12pt]{article}
\usepackage[margin=1in]{geometry}
\usepackage[utf8]{inputenc}
\usepackage{csquotes}
\usepackage{mathtools}
\usepackage{amsthm,amssymb,amsmath,mathabx}
\usepackage{accents}
\newtheorem{theorem}{Theorem}

\usepackage{hhline}

\newtheorem{corollary}[theorem]{Corollary}

\newtheorem{remark}{Remark}
\newtheorem{problem}{Problem}
\usepackage{tabularx}
\usepackage[hidelinks]{hyperref}
\usepackage{diagbox}

\newcommand{\A}{\mathcal{A}}
\newcommand{\F}{\mathcal{F}}
\renewcommand{\L}{\mathcal{L}}

\usepackage{xcolor}

\title{An explicit condition for boundedly supermultiplicative subshifts}
\author{Vuong Bui\thanks{LIRMM, Universit\'e de Montpellier, CNRS, 161 Rue Ada, 34095, Montpellier, France} \thanks{UET, Vietnam National University, Hanoi, 144 Xuan Thuy Street, Hanoi, 100000, Vietnam} \and Matthieu Rosenfeld\footnotemark[1]}
\date{}
\begin{document}
\maketitle
\begin{abstract}
    We study some properties of the growth rate of  $\L(\A,\F)$, that is, the language of words over the alphabet $\A$ avoiding the set of forbidden factors $\F$. We first provide a sufficient condition on $\F$ and $\A$ for the growth of $\L(\A,\F)$ to be boundedly supermultiplicative.  That is, there exist constants $C>0$ and $\alpha\ge0$, such that for all $n$, the number of words of length $n$ in $\L(\A,\F)$ is between $\alpha^n$ and $C\alpha^n$. In some settings, our condition provides a way to compute $C$, which implies that $\alpha$, the growth rate of the language, is also computable whenever our condition holds.

    We also apply our technique to the specific setting of power-free words where the argument can be slightly refined to provide better bounds. Finally, we apply a similar idea to $\F$-free circular words and in particular we make progress toward a conjecture of Shur about the number of square-free circular words.
\end{abstract}
\section{Introduction}
We study some properties of the growth rate of languages defined by sets of forbidden factors. 
Given an alphabet $\A$ and a set of factors $\F\subseteq\A^+$, we denote by $\mathcal{L}(\A,\F)$ the set of words over $\A$ that do not contain any occurrences of factors from $\F$. For all $n$, we let $\mathcal{L}_n(\A,\F)$ be the set of words from $\mathcal{L}(\A,\F)$ of length $n$.
The \emph{growth rate} $\alpha(\A,\F)$ of $\mathcal{L}(\A,\F)$ is given by
\begin{equation*}
    \alpha(\A,\F) = \lim_{n\rightarrow\infty} |\mathcal{L}_n(\A,\F)|^{1/n}\,.
\end{equation*}
The so-called \emph{topological entropy} of the corresponding subshift is given by $\log \alpha(\A,\F)$ (see for instance \cite{Rosenfeld2022Apr}). 

A sequence $(s_n)_{n\in\mathbb{N}}$ is \emph{submultiplicative} if $s_{n+m}\le s_ns_m$ for every $n,m\ge 1$.
Fekete's lemma states that for every submultiplicative nonnegative sequence  $(s_n)_{n\in \mathbb{N}}$, we have $\lim\limits_{n\to\infty} \sqrt[n]{s_n}=\inf\limits_n \sqrt[n]{s_n}$.
The sequence $\left(|\mathcal{L}_n(\A,\F)|\right)_{n\in\mathbb{N}}$ is submultiplicative, so for every $n$,
\begin{equation}\label{submultiplicative}
    |\mathcal{L}_n(\A,\F)|\ge\alpha(\A,\F)^n\,.
\end{equation}

We also later use the version for a supermultiplicative positive sequence $(s_n)_{n\in\mathbb{N}}$, which states that $\lim\limits_{n\to\infty} \sqrt[n]{s_n}=\sup\limits_n \sqrt[n]{s_n}$ (a sequence $s_n$ is supermultiplicative if $s_{n+m}\ge s_ns_m$ for every $n,m\ge 1$). The two versions can be seen to be dual forms of each other. Although it is quite obvious by definition that $\left(|\mathcal{L}_n(\A,\F)|\right)_{n\in\mathbb{N}}$ is submultiplicative, some insight is required to prove a weak form of supermultiplicativity of the same sequence.



When the set $\F$ is finite, the growth rate is given by the spectral radius of the matrix associated with the finite automaton that accepts this language.
Using elementary linear algebra, this implies all sorts of nice properties. 
For instance, $\alpha(\A,\F)$ is computable and even algebraic.
It also entails more precise versions of \eqref{submultiplicative}: There are computable constants $c,t\ge0$ such that $|\mathcal{L}_n(\A,\F)|\sim c n^t \alpha(\A,\F)^n$. 

However, when $\F$ is infinite, the situation is obviously more complicated and more interesting. The motivation behind this work stems from the following result.
\begin{theorem}[\cite{miller,Ochem2016Feb,Rosenfeld2022Apr}\footnote{To be precise the proof in \cite{miller} only implies that $\alpha(\A,\F)\ge 1$ under these hypotheses.}]\label{boundOnGrowth}
    Let $\A$ be an alphabet and let $\F\subseteq\A^+$ be a set of factors.
    Suppose there exists $\beta>1$ such that 
    \begin{equation}\label{mainCondition}
        |\A| \ge \beta+\sum_{f\in \mathcal F} \beta^{1-|f|}\,,
    \end{equation}
    then for all $n\ge0$,
    $$|\mathcal{L}_{n+1}(\A,\F)|\ge \beta|\mathcal{L}_{n}(\A,\F)|\,.$$
    In particular, this implies
    $$\alpha(\A,\F)\ge \beta\,.$$
\end{theorem}
The different proofs of this result rely on similar ideas that we also adapt in this article. The first occurrence of this idea seems to be by Golod and Šafarevič in an algebraic context \cite{golod}. Almost forty years later, Bell and Goh discovered that the same idea could be used in combinatorics on words \cite{Bell2007Sep}. Miller's proof seems to be independent of Golod and Šafarevič's article, but relies on a similar idea. A similar idea was also independently rediscovered by Kolpakov \cite{Kolpakov2007, Kolpakov2007Dec}, and improved by Shur \cite{Shur2010Jul,Shur2014Feb} (see \cite{Shur2012Nov}, for a survey on this topic in relation to the growth of power-free languages).

The readers should note that \eqref{mainCondition} as well as other equations in the article involve $|f|$ for $f\in \F$. A small issue is that different sets $\F$ of forbidden factors may give the same language $\L(\A,\F)$. Therefore, most bounds in this article depend on the choice of $\F$. However, it is not hard to see that we can always choose $\F$ to be minimal, in the sense that each factor in $\F$ avoids all the other factors in $\F$. For a given language $\L$, this choice of $\F$ always gives the best bounds among all the choices. However, there are other ways to improve this choice of $\F$. For instance, it is known that the growth of infinitely extendable words from $\L$ is the same as the growth of $\L$ which can be used to obtain even better bounds.

The condition of this theorem is general enough to solve open problems about specific languages (see for instance, \cite{Ochem2016Feb,Rosenfeld2022Apr}). 
Extensions of the technique behind this proof were also used by the second author to study other related problems, in particular in the context of avoidability of powers \cite{Rosenfeld2024May,Rosenfeld2022Sep,Rosenfeld2021Oct}.
However, for any particular language, it is hard to say how good the lower bound is, and it says nothing about the computability of $\alpha(\A,\F)$.
However, by making the inequality \eqref{mainCondition} strict, we can actually prove the following weak form of supermultiplicativity for $(|\L_n(\A,\F)|)_n$, whose corollaries include a lower bound on $\alpha(\A,\F)$.
\begin{theorem}\label{intromult}
    Let $\A$ be an alphabet and let $\F\subseteq\A^+$ be a set of factors.
    Suppose there exists $\beta>1$ such that   
    \begin{equation*}
        |\A| > \beta+\sum_{f\in \F} \beta^{1-|f|}\,,
    \end{equation*}
    then there exists $C>0$ for all $n,m\ge 0$,
    \begin{equation}\label{supermultiplicative}
        |\mathcal{L}_{n+m}(\A,\F)|\ge C\cdot|\mathcal{L}_{n}(\A,\F)|\cdot|\mathcal{L}_{m}(\A,\F)|\,.
    \end{equation}
\end{theorem}
Theorem \ref{intromult} is in fact a corollary of Theorem \ref{thmSupermult}, which also provides an explicit expression $C=1-\sum_{f\in \F} (|f|-1)\beta^{-|f|}$ (the deduction is given in  Corollary \ref{cor-analytic}).
Subshifts with property \eqref{supermultiplicative} are called \emph{boundedly supermultiplicative shifts} in \cite{Baker2015Oct} and this property also plays a central role in \cite{Pavlov2022Apr}. 
The condition of our approach is more explicit in terms of constants than the condition given by Pavlov \cite[Theorem 4.3]{Pavlov2022Apr}.
The difference in techniques allows us to conclude for all $n,m$ instead of only sufficiently large $n,m$ as in Pavlov's approach. 
We discuss this comparison in more detail in Section \ref{secCompare}.

Inequalities \eqref{submultiplicative} and \eqref{supermultiplicative} imply that 
\begin{equation}\label{eq:constant-factor}
|\mathcal{L}_n(\A,\F)|^{1/n}\cdot C^{1/n}\le\alpha(\A,\F)\le |\mathcal{L}_n(\A,\F)|^{1/n}\,.
\end{equation}
Under the assumption that $\mathcal{L}_n(\A,\F)$ and $C$ can be computed (or that some positive lower bound on $C$ can be computed), the growth rate $\alpha(\A,\F)$ can also be computed.

With more information about the set $\F$, one can improve the conclusions of Theorem \ref{boundOnGrowth}. In particular, in the case of $p$-powers or patterns, we have a stronger version of this theorem \cite{Ochem2016Feb}. Similarly, we can also obtain a stronger version of our result in the setting of avoidability of $p$-powers. In particular, it allows us to prove \eqref{eq:constant-factor} for the language of square-free words over an alphabet of size $5$ for some constant $C$. The particular attention given to $p$-powers is motivated by the fact that this subject received a lot of attention in combinatorics on words since the seminal work of Thue \cite{Thue1}.

We use a similar technique once more to give a partial resolution to a conjecture by Shur about the growth of the number of circular-square-free words \cite{Shur2010Oct}. Circular words can be seen as cyclically ordered finite sequences of letters. That is, we interpret the word as written on a circle, and we read factors along this circle. More formally, we say that two words $u$ and $v$ are conjugate if there are words $s,t$ such that $u=st$ and $v=ts$. The \emph{cyclic word} $\accentset{\circ}{u}$, is the set of conjugates of the word $u$, and we say that $\accentset{\circ}{u}$ avoids a factor $f$ if for all conjugates $v\in \accentset{\circ}{u}$, we have $v$ avoids $f$. 
In \cite{Shur2010Oct}, Shur proved that the number of square-free circular words of length $n$ and minimal squares of length $2n$ (squares that do not contain another square as a proper factor) grows at the same rate \cite[Proposition 1]{Shur2010Oct}.
In the same article, Shur wrote:
\begin{quote}
    From computer experiments, we learn that the growth rate for the set of ternary square-free circular words is approximately $1.3$.
    This growth rate obviously cannot exceed the growth rate for the set of ternary square-free ordinary words.  For the latter one, a nearly exact value is now known \cite{Shur2009}. 
    It is $1.30176\ldots$, so it is quite intriguing whether the two considered growth rates coincide.
\end{quote}
The question of estimating the number of ternary square-free circular words was also explicitly asked in \cite{Currie2020May}. 
 We prove in Theorem \ref{circularsquarefreethm} that Shur's conjecture is true for any alphabet of size at least $5$, but we leave open the cases of the alphabets of size $4$ and $3$. The question of $\F$-free circular words has also been studied in a context more general than squares \cite{Currie2003Jun, Currie2021Jan}, so we also provide a general sufficient condition on $\F$ in Theorem \ref{circularFfreethm} that ensures the existence of $\F$-free circular words.

As already mentioned, all of these proofs rely on similar ideas, some of which can already be found in the proof of Theorem \ref{boundOnGrowth}. We start Section \ref{mainssection} by reproducing the proof of Theorem \ref{boundOnGrowth}. We then state and prove Theorem \ref{thmSupermult}, and we discuss a few consequences and applications. 
In Section \ref{powersection}, we prove a version of this result tailored to the context of avoidability of $p$-powers.
Finally, in Section \ref{Shurconjsection} we provide our partial solution to  Shur's conjecture and our general condition for the existence of $\F$-free circular words.

\section{Supermultiplicativity in the general case}\label{mainssection}
In this section, we first reproduce the proof of Theorem \ref{boundOnGrowth} given in \cite{Rosenfeld2022Apr}. 
Our main motivation for reproducing this proof here is that the other main proofs of this article rely on similar ideas.
We then state and prove Theorem \ref{thmSupermult} which is our sufficient condition for bounded supermultiplicativity.
We then discuss different consequences of this result.

For the rest of the article, given two languages $L_1$ and $L_2$, we denote by $L_1\cdot L_2$ the concatenation of the two languages, that is,  $L_1\cdot L_2= \{uv: u\in L_1,v\in L_2\}$.

\begin{proof}[Proof of Theorem \ref{boundOnGrowth}]
    Let $\A$ be an alphabet, $\F\subseteq\A^+$ be a set of factors and $\beta>1$ such that
    \begin{equation*}
        |\A| \ge \beta+\sum_{f\in \F} \beta^{1-|f|}\,.
    \end{equation*}
    We prove by strong induction on $n$ that for all $n$, $|\mathcal{L}_{n+1}(\A,\F)|\ge \beta|\mathcal{L}_{n}(\A,\F)|$.
    
    Assume that for all $0\le i<n$, $|\mathcal{L}_{i+1}(\A,\F)|\ge \beta|\mathcal{L}_{i}(\A,\F)|$, and we will prove that it is true for $i=n$ as well.
    
   The induction hypothesis implies that for all $j<n$, 
    \begin{equation}\label{FromIHp}
        |\mathcal{L}_{n-j}(\A,\F)|\le \frac{|\mathcal{L}_{n}(\A,\F)|}{\beta^j}\,.
    \end{equation}

    We let $\mathfrak{B}$ be the set of words not in $\L_{n+1}(\A,\F)$ obtained by concatenating a letter to a word from $\L_n(\A,\F)$, that is,
    \begin{align*}
        \mathfrak{B}
        &=\{ua: u\in \mathcal{L}_{n}(\A,\F), a\in \A, ua \text{ does not avoid } \F\}\\
        &=(\mathcal{L}_{n}(\A,\F)\cdot \A) \setminus \mathcal{L}_{n+1}(\A,\F)\,.
    \end{align*}
It follows from $\L_{n+1}(\A,\F)\subseteq \L_n(\A,\F)\cdot \A$ that
    \begin{equation}\label{Binequality}
        |\mathcal{L}_{n+1}(\A,\F)| = |\A|\cdot|\mathcal{L}_{n}(\A,\F)|-|\mathfrak{B}|\,.
    \end{equation}
    For all $f\in \F$, we let $\mathfrak{B}_f$ be the set of words from $\mathfrak{B}$ that ends with $f$ as a suffix. By definition, $\mathfrak{B}=\bigcup_{f\in \F} \mathfrak{B}_f$ which implies 
    \begin{equation}\label{partitioninBf}
        |\mathfrak{B}|\le \sum_{f\in \F} |\mathfrak{B}_f|\,.
    \end{equation}
    
    For all $f\in \F$, and $v\in\mathfrak{B}_f$, there exists $u\in \mathcal{L}_{n+1-|f|}(\A,\F)$ such that $v=uf$. 
    That is, $\mathfrak{B}_f\subseteq\{uf: u\in \mathcal{L}_{n+1-|f|}(\A,\F)\}$. If follows from \eqref{FromIHp}, that for all $f$, 
    \begin{equation*}
        |\mathfrak B_f|\le |\L_{n+1-|f|}(\A,\F)|\le \frac{|\L_{n}(\A,\F)|}{\beta^{|f|-1}}\,.
    \end{equation*}
    Note that although $\mathcal L_{n+1-|f|}$ is undefined when $n+1< |f|$, the final bound still holds (as the set $\mathfrak B_f$ is empty).
    
    Now from \eqref{partitioninBf},
    \[
        |\mathfrak{B}|\le \sum_{f\in \F} \frac{|\L_{n}(\A,\F)|}{\beta^{|f|-1}}.
    \]
Finally,  using \eqref{Binequality} and \eqref{mainCondition},
    \begin{align*}
        |\L_{n+1}(\A,\F)|
        &\ge |\A||\L_n(\A,\F)| - \sum_{f\in \F} \frac{|\L_{n}(\A,\F)|}{\beta^{|f|-1}} \\
        &= |\L_n(\A,\F)|\left(|\A| - \sum_{f\in \F} \beta^{1-|f|}\right) \\
        &\ge \beta|\L_n(\A,\F)|\,,
    \end{align*}
    which concludes our proof.
\end{proof}

\subsection{Bounded supermultiplicativity of \texorpdfstring{$(|\mathcal{L}_n(\A,\F)|)_{n\ge0}$}{the number of factors}}
We now provide a sufficient condition that implies bounded supermultiplicativity of the sequence $\left(|\mathcal{L}_{n}(\A,\F)|\right)_{n\ge0}$.
\begin{theorem}\label{thmSupermult}
    Let $\A$ be an alphabet and let $\F\subseteq\A^+$ be a set of factors.
    Suppose there exists $\beta>1$ such that $|\mathcal L_{n+1}(\A,\F)|\ge\beta |L_n(\A,\F)|$ for all $n\ge 0$ and 
    \begin{equation}\label{secondCondition}
        C\coloneq 1-\sum_{f\in \F} (|f|-1)\beta^{-|f|}>0\,,
    \end{equation}
    then for all $n,m\ge 0$,
    $$|\mathcal{L}_{n+m}(\A,\F)|\ge C\cdot|\mathcal{L}_{n}(\A,\F)|\cdot|\mathcal{L}_{m}(\A,\F)|\,.$$
\end{theorem}
Note that if $C\le0$, then the conclusion of this theorem trivially holds as well, but we restrict this theorem to the meaningful case (this comment also applies to the other results of the article).
\begin{proof}
For all $f\in\F$, we let $\mathfrak{B}_f$ be the set of words from $\L_{n}(\A,\F)\cdot\L_{m}(\A,\F)$ that contain an occurrence of $f$.
Since  $\mathcal{L}_{n+m}(\A,\F) \subseteq \mathcal{L}_{n}(\A,\F)\cdot\mathcal{L}_{m}(\A,\F)$, this implies
\begin{equation*}
    \mathcal{L}_{n+m}(\A,\F) = \mathcal{L}_{n}(\A,\F)\cdot\mathcal{L}_{m}(\A,\F) -\bigcup_{f\in\F} \mathfrak{B}_f\,,
\end{equation*}
hence
\begin{equation}\label{sizeOfLmn}
    |\mathcal{L}_{n+m}(\A,\F)| \ge |\mathcal{L}_{n}(\A,\F)|\cdot|\mathcal{L}_{m}(\A,\F)| -\sum_{f\in\F} |\mathfrak{B}_f|\,.
\end{equation}

For any $f\in\F$ and $w\in \mathfrak{B}_f$, $f$ occurs in $w$ around the concatenation, that is, there exists $i\in\{1,\ldots, |f|-1\}$ such that $w=ufv$ with $|u|=n-i$ and $|v|=m+i-|f|$. By definition, $u\in \mathcal{L}_{n-i}(\A,\F)$ and $v\in \mathcal{L}_{m+i-|f|}(\A,\F)$. Hence,
\begin{equation*}
    \mathfrak{B}_f \subseteq \bigcup_{i=1}^{|f|-1}\left\{ufv : u\in \mathcal{L}_{n-i}(\A,\F), v\in \mathcal{L}_{m+i-|f|}(\A,\F) \right\}\,.
\end{equation*}
By the condition $|\mathcal L_{n+1}(\A,\F)|\ge\beta |\mathcal L_n(\A,\F)|$ for all $n\ge 0$, we have, for all $f\in \F$,
\begin{align*}
    |\mathfrak{B}_f|
    &\le \sum_{i=1}^{|f|-1}|\mathcal{L}_{n-i}(\A,\F)|\cdot|\mathcal{L}_{m+i-|f|}(\A,\F)|\\
    &\le \sum_{i=1}^{|f|-1}\frac{|\mathcal{L}_{n}(\A,\F)|}{\beta^{i}}\cdot\frac{|\mathcal{L}_{m}(\A,\F)|}{\beta^{|f|-i}}\\
    &= (|f|-1)\cdot\beta^{-|f|}|\mathcal{L}_{n}(\A,\F)|\cdot|\mathcal{L}_{m}(\A,\F)|\,.
\end{align*}
Note that although $\mathcal L_{n-i}$ and $\mathcal L_{m+i-|f|}$ are undefined when $n,m$ are too small, the final \emph{upper} bound still holds. 

Using this in \eqref{sizeOfLmn}, we get
\begin{align*}
    |\mathcal{L}_{n+m}(\A,\F)| 
    &\ge  |\mathcal{L}_{n}(\A,\F)|\cdot|\mathcal{L}_{m}(\A,\F)|\left(1-\sum_{f\in\F} (|f|-1)\beta^{-|f|}\right)\\
    &=C\cdot|\mathcal{L}_{n}(\A,\F)|\cdot|\mathcal{L}_{m}(\A,\F)|\,,
\end{align*}
as desired.
\end{proof}

Remember that for all $\F$, the sequence $(|\mathcal{L}_{n}(\A,\F)|)_{n\ge0}$ is submultiplicative. Fekete's lemma implies that for all $n$,
$|\mathcal{L}_{n}(\A,\F)|\ge \alpha(\A,\F)^n\,.$ 
In general, this bound informs us about the behavior of $(|\mathcal{L}_{n}(\A,\F)|)_{n\ge0}$, but it is also useful to compute upper bounds on $ \alpha(\A,\F)$.
Theorem \ref{thmSupermult} implies a form of supermultiplicativity that plays a symmetric role.

\begin{corollary}\label{corSupermult}
    Under the conditions of Theorem \ref{thmSupermult}, we have for all $n\ge0$,
    \begin{equation*}
    \alpha(\A,\F)^n\le|\mathcal{L}_{n}(\A,\F)|\le C^{-1}\cdot\alpha(\A,\F)^n\,.
    \end{equation*}
\end{corollary}
\begin{proof}
By Theorem \ref{thmSupermult}, we have for all $n,m\ge0$,
$$|\mathcal{L}_{n+m}(\A,\F)|\ge C\cdot|\mathcal{L}_{n}(\A,\F)|\cdot|\mathcal{L}_{m}(\A,\F)|\,.$$
Consider the sequence $(L'_n)_{n\ge0}$ such that for all $n$, $L'_n =  C\cdot|\mathcal{L}_{n}(\A,\F)|$,
then for all $n,m\ge0$,
\[
    L'_{n+m}\ge  L'_n\cdot L'_m\,.
\]
That is, $(L'_n)_{n\ge0}$ is a supermultiplicative sequence. The growth rate of $(L'_n)_{n\ge0}$ is also the growth rate of $|\L_n(\A,\F)|$, which is $\alpha(\A,\F)$ by definition. By Fekete's lemma for (positive) supermultiplicative sequences, we have
for all $n\ge0$, $(L'_n)^{1/n}\le \alpha(\A,\F)$, which implies for all $n$,
$$|\mathcal{L}_{n}(\A,\F)|\le C^{-1}\cdot\alpha(\A,\F)^n\,,$$
as desired.
\end{proof}
Equivalently, the conclusion of this corollary can be rewritten as
\begin{equation*}
(C|\mathcal{L}_{n}(\A,\F)|)^{1/n}\le\alpha(\A,\F)\le|\mathcal{L}_{n}(\A,\F)|^{1/n}\,.
\end{equation*}
This implies that the difference between $\alpha(\A,\F)$ and $|\mathcal{L}_{n}(\A,\F)|^{1/n}$ goes to $0$ as $n$ goes to infinity. Moreover, if we know a positive lower bound on $C$, then for all $\varepsilon>0$, we can find a value $n_\varepsilon$ such that $|\ |\mathcal{L}_{n_\varepsilon}(\A,\F)|^{1/n_\varepsilon}-\alpha(\A,\F)|\le \varepsilon$.
As a direct consequence, we have the following result.
\begin{theorem}
    There is an algorithm that solves the following problem for any set of forbidden factors $\F$ that satisfy the conditions of Theorem \ref{thmSupermult}:\\
    \begin{tabularx}{\textwidth}{@{\hspace{\parindent}} l X c}
    \textbf{Input:} &  An oracle recognizing $\F$,
 a real $C$ such that $0<C\le1-\sum_{f\in \F} (|f|-1)\beta^{-|f|}$,
       and a positive real $\varepsilon$.\\
    \textbf{Output:} & A real number $\alpha'$ such that
    $|\alpha'-\alpha(\A,\F)|\le \varepsilon\,.$
  \end{tabularx}
\end{theorem}

\subsection{Application}
To illustrate these results, we consider the case where $\F$ contains at most one factor of each length and let $i= \min \{|f|:f\in \F\}$. 
Miller \cite{miller} gave simple sufficient conditions in terms of $i$ and $|\A|$ such that $\alpha(\A,\F)\ge1$. He also proved that these conditions are optimal in the sense that for every $(i,|\A|)$ that does not respect the condition, there is a choice of $\F$ such that  $\alpha(\A,\F)=0$. 
The second author improved this result by providing lower bounds on $\alpha(\A,\F)$ for each of these cases \cite{Rosenfeld2022Sep}. Here, we go further by classifying the pairs $(i,|\A|)$ such that $(\L_n(\A,\F))_{n\ge1}$ is necessarily boundedly submultiplicative.
\begin{corollary}\label{onepersize}
    Let $\F$ be a set of factors that contains at most one factor of each length, and let $i= \min \{|f|:f\in \F\}$. 
    Suppose there exists $\beta>1$ such that
    \begin{equation}\label{cond1corollary}
        |\A|-\frac{\beta^{2-i}}{\beta-1} \ge \beta\,,
    \end{equation}
    and 
    \begin{equation}\label{cond2corollary}
        C\coloneq 1-\frac{\beta ^{1-i}(1+(\beta-1)(i-1))}{(\beta-1)^2}>0\,,
    \end{equation}
    then for all $n\ge0$,
    \begin{equation}\label{conscorollary}
    \alpha(\A,\F)^n\le|\mathcal{L}_{n}(\A,\F)|\le C^{-1}\cdot\alpha(\A,\F)^n\,.
    \end{equation}
\end{corollary}
\begin{proof}
    We use the fact that
    \begin{equation*}
        \sum_{\ell\ge i} \beta^{1-\ell}= \frac{\beta^{2-i}}{\beta-1}\,,
    \end{equation*}
    and
    \begin{equation*}
        \sum_{\ell\ge i} (\ell-1)\beta^{-\ell}= \frac{\beta ^{1-i}(1+(\beta-1)(i-1))}{(\beta-1)^2}\,,
    \end{equation*}
    to apply Corollary \ref{corSupermult}. 
\end{proof}
For different small values of $|\A|$ and $i$, we provide in Table \ref{tablebetaC}, values of $\beta$ and a lower bound on $C$ that satisfy the condition of Corollary \ref{onepersize}. 
The first three pairs of cells of this table are empty, because there is no $\beta$ satisfying equation \eqref{cond1corollary}. In these three cases the corresponding subshift is empty for some choices of $\F$ as already argued in \cite{miller,Rosenfeld2022Apr}.
\begin{table}[t]
    \begin{center}
        \begin{tabular}{ |c||l|l||l|l||l|l|} 
            \hline
            \backslashbox{$i$}{$|\A|$} & \multicolumn{2}{c||}{$2$} & \multicolumn{2}{c||}{$3$}& \multicolumn{2}{c|}{$4$}\\\hhline{|=||=|=||=|=||=|=|}
            $2$ & && $\beta=2$& $C=0$& $\beta=3.6$& $C\ge 0.85$\\\hline
            $3$ & && $\beta=2.75$& $C=0.8$&  $\beta=3.91$& $C\ge 0.94$\\\hline
            $4$ & && $\beta=2.9$& $C\ge 0.92$& $\beta=3.97$& $C\ge 0.98$\\\hline
            $5$ & $\beta=1.72$& $C\ge 0.14$& $\beta=2.97$& $C\ge 0.97$& $\beta=3.99$& $C\ge 0.99$\\\hline
            $6$ & $\beta=1.91$& $C\ge 0.73$& $\beta=2.99$& $C\ge 0.98$& $\beta=3.99$& $C\ge 0.998$\\\hline
        \end{tabular}
        \caption{Values of $\beta$ and a lower bound (obtained by rounding down) on the corresponding $C$ from Corollary \ref{onepersize}, for small values of $|\A|$ and $i$.\label{tablebetaC}}
    \end{center}
\end{table}
As we will see, the case $|\A|=3$ and $i=2$ is also optimal. Over the alphabet $\A=\{0,1,2\}$, consider the set 
$\F=01^*2:=\{02, 012, 0112, 01112, \ldots\}.$
It is not hard to see that $\mathcal{L}(\A,\F)$ is exactly the set of words that contain no $2$ after a $0$, that is, $\mathcal{L}(\A,\F)= \{u0v: u\in\{1,2\}^*,v\in\{0,1\}^*\}\cup\{1,2\}^*$. 
One easily verifies that for all $n$,  $|\mathcal{L}_n(\A,\F)|=(n+2)2^{n-1}$.
So there is no $C>0$ for which \eqref{conscorollary} holds for this particular choice of $\F$. The choice of $\F$ was guided by the proofs.
Indeed, $\F$ was constructed to ensure that the different crucial inequalities from the proofs are tight.

When $C$ is known and if we can compute $\mathcal{L}_{n}(\A,\F)$ for any $n$, then we can compute arbitrarily good approximations of $\alpha(\A,\F)$. 
Therefore, a simple inspection of Table \ref{tablebetaC} implies the following.
\begin{corollary}
Let $\A$ be an alphabet and  $i$ be an integer such that   
    \begin{itemize}
        \item $|\A| \ge4$ and $i\ge2$,
        \item or $|\A| = 3$ and $i\ge3$,
        \item or $|\A| = 2$ and $i\ge5$.
    \end{itemize}
Then there is an algorithm that solves the following problem:\\
    \begin{tabularx}{\textwidth}{@{\hspace{\parindent}} l X c}
    \textbf{Input:} &  An oracle recognizing a set $\F$  that contains at most one factor of each length and such that $\min \{|f|:f\in \F\}\ge i$
       and a positive real $\varepsilon$.\\
    \textbf{Output:} & A real number $\alpha'$ such that
    $        |\alpha'-\alpha(\A,\F)|\le \varepsilon\,.$
  \end{tabularx}
\end{corollary}

For the four missing cases, it is not clear whether $\alpha(\A,\F)$ is computable (but it would not be surprising if it could be computed with other techniques).

\subsection{Analytic considerations}\label{Analyticsubsec}
We say that $\F$ is \emph{non-trivial} if it contains at least one factor of length at least $2$.
Given $\F$, we define the function $\omega$ such that
\begin{equation*}
    \omega(x)=x+\sum_{f\in\F}x^{1-|f|}\,,
\end{equation*}
for all $x$ such that the infinite series is well-defined.
The main condition on $\beta$ from \eqref{mainCondition} in Theorem \ref{boundOnGrowth} can be rewritten as
\begin{equation*}
    \omega(\beta)\le|\mathcal{A}|\,.
\end{equation*}
 The derivative of $\omega$ is given by
 \begin{equation*}
    \omega'(x)=1-\sum_{f\in\F}(|f|-1)x^{-|f|}\,,
\end{equation*}
for all $x$ such that the infinite series is well-defined.
The second condition of Theorem \ref{thmSupermult}, equation \eqref{secondCondition}, is the same as
\begin{equation*}
    C\coloneq\omega'(\beta)> 0\,.
\end{equation*}

The nature of the proofs of these two theorems are quite similar, but it is not completely clear why the derivative of the first condition gives the second condition. We now try to provide a fragment of explanation for this. The second derivative of $\omega$ is, for all $x$,
\begin{equation*}
    \omega''(x)=\sum_{f\in\F}(|f|-1)|f|x^{-|f|-1}\,.
\end{equation*}
As long as $\F$ is non-trivial this function is positive for $\beta>1$, which implies that $\omega'$ is increasing. This means that the set of solutions $\{\beta>1:\omega(\beta)\le|\mathcal{A}|\}$ is either empty or an interval. So there are four possibilities:
\begin{itemize}
    \item $\{\beta>1:\omega(\beta)\le|\mathcal{A}|\}=\emptyset$,
    \item $\{\beta>1:\omega(\beta)\le|\mathcal{A}|\}=\{x\}$, in which case $\omega'(x)=0$,
    \item $\{\beta>1:\omega(\beta)\le|\mathcal{A}|\}$ is a proper interval $I$ closed on the right at $y=\max I$, in which case, $\omega'(y)>0$.
\end{itemize}
In other words, we can apply Theorem \ref{thmSupermult} with some $C>0$ if and only if there are at least two $\beta$ that satisfy the conditions of Theorem \ref{boundOnGrowth}. For instance, we can deduce the following form of Corollary \ref{corSupermult}, which is equivalent to the statement of Theorem \ref{intromult} given in the introduction.
\begin{corollary}\label{cor-analytic}
    Let $\beta \ge1$ be such that $\omega(\beta)<|\A|$, then there exists some $C>0$ such that
    \begin{equation*}
    \alpha(\A,\F)^n\le|\mathcal{L}_{n}(\A,\F)|\le C^{-1}\cdot\alpha(\A,\F)^n\,.
    \end{equation*}
\end{corollary}

The conditions of this version are easier to verify, however, it does not provide an explicit value for $C$ other than $C\coloneq\omega'(\beta)$.
Tools from analytic combinatorics might be helpful in finding a more satisfying explanation to why the derivative plays a role there, which we were not able to do.

\subsection{Comparison with \texorpdfstring{\cite[Theorem 4.3]{Pavlov2022Apr}}{Theorem 4.3}}
\label{secCompare}
For the sake of comparison, we reproduce the statement Theorem 4.3 from \cite{Pavlov2022Apr} using our notation.
\begin{theorem}[{\cite[Theorem 4.3]{Pavlov2022Apr}}]\label{pavlovresult}
If there exists $\beta<\frac{\alpha(\A,\F)^2}{|\A|}$ for which
\begin{equation}\label{PavlovCondition}
    \sum_{f\in\F} |f|\beta^{-|f|}<\frac{1}{36}\,,
\end{equation}
then for all sufficiently large $n$,
\begin{equation*}
    |\mathcal L_n(\A,\F)|<4\alpha(\A,\F)^n\,.
\end{equation*}
\end{theorem}
Pavlov uses the condition $\beta<\frac{\alpha(\A,\F)^2}{|\A|}$ to ensure that the set of integers $n$ so that $|\mathcal L_n(\A,\F)|\ge \beta^t |L_{n-t}(\A,\F)|$ for every $t$ has lower density at least $\frac12$. Meanwhile, our condition $|\A| \ge \beta+\sum_{f\in \F} \beta^{1-|f|}$ is there to ensure that $|\mathcal L_{n+1}(\A,\F)|\ge\beta|\mathcal L_n(\A,\F)|$ for every $n$. That is why Theorem \ref{pavlovresult} only holds for sufficiently large $n$, while our conclusion holds for every $n$. Also note that our condition is more explicit on $\beta$ as we do not reference $\alpha(\A,\F)$. On the other hand, the strict inequality $|\A| > \beta+\sum_{f\in \F} \beta^{1-|f|}$ in our condition is equivalent to
\[
    \frac{|\A|}{\beta}-1 > \sum_{f\in \F} \beta^{-|f|},
\]
which is in general weaker than the second condition of Theorem \ref{pavlovresult}.
We would emphasize that another merit of our work lies in the approach where we give a relatively direct proof of the supermultiplicativity, while Pavlov's proof still uses the submultiplicativity $|\mathcal L_{n+m}(\A,\F)|\le |\mathcal L_n(\A,\F)||\mathcal L_m(\A,\F)|$ and bounds $|\mathcal L_n(\A,\F)|$ by a constant of $\alpha(\A,\F)^n$ for certain values of $n$ by a variant of Miller's method \cite{miller}. Also note that the constants given by our result are better in many cases, but replacing Miller's techniques with those from Theorem \ref{boundOnGrowth} would improve the constants in Theorem \ref{pavlovresult} as well.

\section{Power versions}\label{powersection}
For all $p>1$ a non-empty word $w$ is a \emph{$p$-power} if it can be written $w=u^{\lfloor p\rfloor}v$ with $v$ a prefix of $u$ and $\frac{|w|}{|u|}\ge p$. 
The period of the $p$-power is $u$ (or $|u|$ by abuse of notation). 
A word is \emph{$p$-power-free} if none of its factors is a $p$-power. 
For all real $p>1$, we let $\mathcal{L}(\A,p)$ be the set of $p$-power-free words over alphabet $\A$.

Theorem \ref{boundOnGrowth} and Theorem \ref{thmSupermult} can be applied to the setting of $p$-power-free words. However, in this setting, there is a more efficient way to use the same ideas. In this section, we provide a version of  these two theorems tailored to $p$-power-free words. 
The proof of the following result resembles the proof of Theorem \ref{boundOnGrowth}. A similar result was also provided in \cite{Rosenfeld2021Oct}.
\begin{theorem}\label{boundOnGrowthPow}
Let $\A$ be an alphabet and let $p>1$.
Suppose there exists $\beta>1$ such that 
\begin{equation}\label{mainConditionPow}
    |\A|-\sum_{\ell\ge1} \beta^{1-\lceil \ell(p-1)\rceil} \ge \beta\,,
\end{equation}
then for all $n\ge0$,
$$|\mathcal{L}_{n+1}(\A,p)|\ge \beta|\mathcal{L}_{n}(\A,p)|\,.$$
\end{theorem} The main difference with Theorem \ref{boundOnGrowth} is that we group all the $p$-powers of period $\ell$ together in $\mathfrak{B}_\ell$. Then in order to bound the size of $\mathfrak{B}_\ell$, we do not forget the entire $p$-power (which would give $|\mathfrak{B}_\ell|\le |\A|^\ell |\mathcal{L}_{n+1-\lceil \ell p\rceil}(\A,p)|$), but only a fraction of the $p$-power to obtain the bound \eqref{boundonpower} which is more optimal.
\begin{proof}
    We proceed by strong induction on $n$. 
    Assuming $|\mathcal{L}_{i+1}(\A,p)|\ge \beta|\mathcal{L}_{i}(\A,p)|$ for all $0\le i<n$, we will prove that it is true for $i=n$ as well. (Note that the base case is trivial.)
    
    The induction hypothesis implies that for all $j<n$, 
    \begin{equation}\label{FromIHpPow}
        |\mathcal{L}_{n-j}(\A,p)|\le \frac{|\mathcal{L}_{n}(\A,p)|}{\beta^j}\,.
    \end{equation}

    We let $\mathfrak{B}$ be the set of words not in $\mathcal{L}_{n+1}(\A,p)$ obtained by concatenating a letter to a word from $\mathcal{L}_{n}(\A,p)$, that is,
    \begin{align*}
    \mathfrak{B}
    &=\{ua: u\in \mathcal{L}_{n}(\A,p), a\in \A, ua \not\in\mathcal{L}_{n+1}(\A,p)\}\\
    &=(\mathcal{L}_{n}(\A,p)\cdot \A) \setminus \mathcal{L}_{n+1}(\A,p)\,.
    \end{align*}
    It follows from $\L_{n+1}(\A,p)\subseteq \L_n(\A,p)\cdot \A$ that
    \begin{equation}
        |\mathcal{L}_{n+1}(\A,p)| = |\A|\cdot|\mathcal{L}_{n}(\A,p)|-|\mathfrak{B}|\,.
    \end{equation}
    For all $\ell\ge1$, we let $\mathfrak{B}_\ell$ be the set of words from $\mathfrak{B}$ that ends with a $p$-power of period $\ell$. By definition, $\mathfrak{B}=\bigcup_{\ell\ge1} \mathfrak{B}_\ell$ which implies 
    \begin{equation}\label{partitioninBfPow}
        |\mathfrak{B}|\le \sum_{\ell\ge1} |\mathfrak{B}_\ell|\,.
    \end{equation}
    
    For all $\ell\ge1$ and $v\in\mathfrak{B}_\ell$, by definition, $v$ ends with a $p$-power $w$ of period $\ell$, so $|w|\ge p\ell$. On the other hand, if we remove the last letter of $v$ we obtain a $p$-power-free word, which means that $|w|-1<p\ell$. This implies $|w|=\lceil p\ell\rceil$.
    Hence, there exists $u\in \mathcal{L}_{n+1-\lceil(p-1)\ell\rceil}(\A,p)$ and $x\in\mathcal{A}^{\lceil(p-1)\ell\rceil}$ such that $v=ux$ and the period of the $p$-power is the suffix of length $\ell$ of $u$. Now, given $n$, $\ell$ and $u$ there is a unique $x$ such that $ux$ ends with a $p$-power of period $\ell$ and length $\lceil p\ell\rceil$. Hence, for all $\ell\ge1$, $|\mathfrak{B}_\ell|\le |\mathcal{L}_{n+1-\lceil(p-1)\ell\rceil}(\A,p)|$.
     If follows from \eqref{FromIHpPow}, that for all $\ell$, 
    \begin{equation}\label{boundonpower}
        |\mathfrak B_\ell|
        \le |\L_{n+1-\lceil(p-1)\ell\rceil}(\A,p)|
        \le \frac{|\L_{n}(\A,p)|}{\beta^{\lceil(p-1)\ell\rceil-1}}\,.
    \end{equation}
    Now from \eqref{partitioninBfPow},
    \[
        |\mathfrak{B}|\le \sum_{\ell\ge1} \frac{|\L_{n}(\A,p)|}{\beta^{\lceil(p-1)\ell\rceil-1}}.
    \]
    In total,
    \begin{align*}
        |\L_{n+1}(\A,p)|&\ge |\A||\L_n(\A,p)| - \sum_{\ell\ge1} \frac{|\L_{n}(\A,p)|}{\beta^{\lceil(p-1)\ell\rceil-1}} \\
        &= |\L_n(\A,p)|\left(|\A| - \sum_{\ell\ge1} \beta^{1-\lceil(p-1)\ell\rceil}\right) \\
        &\ge \beta|\L_n(\A,p)|.
    \end{align*}
    We finish the induction step and the conclusion follows.
\end{proof}

We are now ready to give the $p$-power-free word analog of Theorem \ref{thmSupermult}.
\begin{theorem}\label{thmSupermultPow}
    Let $\A$ be an alphabet and let $p>1$.
    Suppose there exists $\beta>1$ such that $|\mathcal L_{n+1}(\A,p)|\ge\beta|\mathcal L_n(\A,p)|$ for all $n\ge 0$ and let
    \begin{equation*}
        C\coloneq1-\sum_{\ell\ge1} \frac{\lceil p\ell\rceil-1}{\beta^{\lceil (p-1)\ell\rceil}}>0\,,
    \end{equation*}
    then for all $n,m\ge0$,
    $$|\mathcal{L}_{n+m}(\A,p)|\ge C\cdot|\mathcal{L}_{n}(\A,p)|\cdot|\mathcal{L}_{m}(\A,p)|\,.$$  
\end{theorem}

\begin{proof}
    We proceed by induction on the sum $n+m$. The base case is $n+m=0$, that is $n=m=0$. It means that we need $C\le 1$. We prove that if the conclusion holds for every pair $n',m'$ with $n'+m'<n+m$, then it also holds for $n,m$.
    
    For every integer $\ell\ge1$, we let $\mathfrak{B}_\ell$ be the set of words in $\L_{n}(\A,p)\cdot\L_{m}(\A,p)$ that contain a $p$-power of period $\ell$ and no shorter $p$-power.
    We have
    \begin{equation*}
        \mathcal{L}_{n+m}(\A,p) = \mathcal{L}_{n}(\A,p)\cdot\mathcal{L}_{m}(\A,p) -\bigcup_{\ell\ge 1} \mathfrak{B}_\ell\,.
    \end{equation*}
    This implies
    \begin{equation}\label{sizeOfLmnpow}
        |\mathcal{L}_{n+m}(\A,p)| \ge |\mathcal{L}_{n}(\A,p)|\cdot|\mathcal{L}_{m}(\A,p)| -\sum_{\ell\ge 1} |\mathfrak{B}_\ell|\,.
    \end{equation}
    
    For any integer $\ell\ge1$ and $w\in \mathfrak{B}_\ell$, the $p$-power in $w$ is due to the concatenation, that is, there exists $i\in\{1,\ldots, \lceil p\ell\rceil -1\}$ such that $w=u(xy)^{\lfloor p\rfloor}xv$ with $|u|=n-i$, $|v|=m+i-\lceil p\ell\rceil$, $|xy|=\ell$ and $|(xy)^{\lfloor p\rfloor}x|=\lceil p\ell\rceil$. This decomposition might not be unique, but by taking the leftmost decomposition (i.e., the shortest $u$), we can ensure that $u'=uxy$ and $v$ are $p$-power free. 
    Given the values of $\ell$, $p$, $i\in\{1,\dots,\lceil p\ell\rceil -1\}$, the words $u'\in \mathcal{L}_{n-i+\ell}(\A,p)$ and $v\in\mathcal{L}_{m+i-\lceil p\ell\rceil}(\A,p)$, there is a unique corresponding     
    $w=u'(xy)^{\lfloor p-1\rfloor}xv\in \mathfrak{B}_\ell$. This implies
    \begin{equation*}
        |\mathfrak{B}_\ell|\le \sum_{i=1}^{\lceil p\ell\rceil-1}|\mathcal{L}_{n-i+\ell}(\A,p)|\cdot|\mathcal{L}_{m+i-\lceil p\ell\rceil}(\A,p)\,.
    \end{equation*}
    By the induction hypothesis and the condition $\mathcal L_{n+1}(\A,p)\ge\beta \mathcal L_n(\A,p)$ for all $n$, we have 
    \[
        |\mathcal{L}_{n-i+\ell}(\A,p)|\cdot|\mathcal{L}_{m+i-\lceil p\ell\rceil}(\A,p)|\le C^{-1}\cdot|\mathcal{L}_{n+m-\lceil p\ell\rceil+\ell}(\A,p)|\le C^{-1}\cdot\beta^{-(\lceil p\ell\rceil - \ell)}\cdot |\mathcal{L}_{n+m}(\A,p)|,
    \]
    for all $i$. This upper bound is independent of $i$ which implies
    \begin{equation*}
        |\mathfrak{B}_\ell|
        \le  \frac{\lceil p\ell\rceil-1}{C\beta^{\lceil (p-1)\ell\rceil}}\cdot |\mathcal{L}_{n+m}(\A,p)|\,.
    \end{equation*}
    Using this in \eqref{sizeOfLmnpow}, we get
    \[
    |\mathcal{L}_{n+m}(\A,p)| \ge |\mathcal{L}_{n}(\A,p)|\cdot|\mathcal{L}_{m}(\A,p)| -\frac{1}{C}\sum_{\ell\ge1} \frac{\lceil p\ell\rceil-1}{\beta^{\lceil (p-1)\ell\rceil}}\cdot |\mathcal{L}_{n+m}(\A,p)|\,.
    \]
    Elementary manipulations yield
    \begin{align*}
        |\mathcal{L}_{n+m}(\A,p)| 
        &\ge \left(1+\frac{1}{C}\sum_{\ell\ge1} \frac{\lceil p\ell\rceil-1}{\beta^{\lceil (p-1)\ell\rceil}}\right)^{-1} \cdot  |\mathcal{L}_{n}(\A,p)|\cdot|\mathcal{L}_{m}(\A,p)|\\
        &=C \cdot  |\mathcal{L}_{n}(\A,p)|\cdot|\mathcal{L}_{m}(\A,p)|\,,
    \end{align*}
    where the last equality is simply the definition $C=1-\sum_{\ell\ge1} \frac{\lceil p\ell\rceil-1}{\beta^{\lceil (p-1)\ell\rceil}}$.
    This concludes our induction.   
\end{proof}

\begin{remark}
    The proof of Theorem \ref{thmSupermultPow} is slightly more complicated than the corresponding general result in Theorem \ref{thmSupermult}. Indeed, our proof of Theorem \ref{thmSupermultPow} requires an inductive argument for the cases $p<2$ (for $p\ge2$, we could have directly adapted the proof from Theorem \ref{thmSupermult}). On the other hand, using a similar inductive argument in the proof of Theorem \ref{thmSupermult} does not improve the bounds.
\end{remark}

For the sake of illustration, we provide the following two corollaries of these two theorems for the case of square-free words, that is, $p=2$. 
\begin{corollary}\label{explicitGrowthSquare}
Let $\A$ be an alphabet with $|\A|\ge4$. Let
 \[
    \beta=\frac{|\A|+\sqrt{|\A|^2-4|\A|}}{2}
\]
then for all $n\ge0$,
$$|\mathcal{L}_{n+1}(\A,2)|\ge \beta|\mathcal{L}_{n}(\A,2)|\,.$$
\end{corollary}
\begin{proof}
We have
\[
    \beta=\frac{|\A|+\sqrt{|\A|^2-4|\A|}}{2}=\max\left\{x: x+\frac{x}{x-1}\le|\A|\right\}>1\,,
\]
which implies 
\[
|\A|-\sum_{\ell\ge1}\beta^{1-\ell} = |\A|-\frac{\beta}{\beta-1}\ge\beta\,.
\]
We can then apply Theorem \ref{boundOnGrowthPow} to conclude.
\end{proof}

\begin{corollary}\label{squareFreeSupermult}
Let $\A$ be an alphabet with $|\A|\ge5$ and let 
 $\beta=\frac{|\A|+\sqrt{|\A|^2-4|\A|}}{2}$. 
Then for all $n,m\ge0$,
$$|\mathcal{L}_{n+m}(\A,2)|\ge C\cdot|\mathcal{L}_{n}(\A,2)|\cdot|\mathcal{L}_{m}(\A,2)|\,,$$
where 
\begin{equation*}
    C=1-\frac{1+\beta}{(\beta-1)^2}>0\,.
\end{equation*}  
\end{corollary}
\begin{proof}
We first verify that $C>0$.
For $|\A|=5$, explicit computation of $C$ yields approximately $0.3262$. Since $C$ is an increasing function of $\beta$, which is an increasing function of $|\A|$,  we have $C>0$ for all $|\A|\ge5$. 

By Corollary \ref{explicitGrowthSquare}, we have $|\mathcal{L}_{n+1}(\A,2)|\ge \beta|\mathcal{L}_{n}(\A,2)|$, for all $n\ge0$. Moreover, 
\[
    1-\sum_{\ell\ge1} \frac{ 2\ell-1}{\beta^{\ell}}=1-\frac{1+\beta}{(\beta-1)^2}=C>0\,,
\]
which allows us to apply Theorem \ref{thmSupermultPow} to conclude.
\end{proof}

Note that our constants in Corollary \ref{explicitGrowthSquare} and Corollary \ref{squareFreeSupermult} are by no mean optimal. One can use computer-assisted proofs with automata to obtain better values for $\beta$ in Corollary \ref{explicitGrowthSquare} (see \cite{Shur2012Nov}). However, the constant $C$ in Corollary \ref{squareFreeSupermult} is so far the only one of its kind in the literature. It is not obvious how to use automata to improve the constant $C$ from Corollary \ref{squareFreeSupermult}.

\section{Square-free circular words and \texorpdfstring{$\F$}{F}-free circular words}\label{Shurconjsection}
Let $\accentset{\circ}{\mathcal{L}}(\A,2)$ be the set of words $w$ such that $\accentset{\circ}{w}$ is a square-free circular word over the alphabet $\A$, and for all $n$, $\accentset{\circ}{\mathcal{L}}_n(\A,2)$ consists of those of length $n$.

The main idea behind our proof is that there are so many square-free words of a given length that only a fraction of them can fail to be square-free circular words.

\begin{theorem}\label{circularsquarefreethm}
Let $\A$ be an alphabet. Suppose there exists $\beta>1$ such that $|\mathcal L_{n+1}(\A,2)|\ge\beta|\mathcal L_n(\A,2)|$ for all $n\ge 0$ and let
    \begin{equation*}
        C\coloneq1-\frac{1+\beta}{(\beta-1)^2}>0\,,
    \end{equation*}
then for all $n\ge0$, 
$$|\mathcal{L}_n(\A, 2)|\ge|\accentset{\circ}{\mathcal{L}}_n(\A,2)|\ge C\cdot|\mathcal{L}_n(\A, 2)|\,.$$
\end{theorem}
\begin{proof}
Let $n\ge0$ and $\mathfrak{B} =\mathcal{L}_n(\A, 2)\setminus\accentset{\circ}{\mathcal{L}}_n(\A,2)$, 
which implies
\begin{equation}\label{mainEqCircular}
    |\accentset{\circ}{\mathcal{L}}_n(\A,2)|=|\mathcal{L}_n(\A, 2)|- |\mathfrak{B}|\,.
\end{equation}

For any $w\in \mathfrak{B}$, $w$ is square-free and there exists words $p,v,s$ such that $w=pvs$, $p$ and $s$ are non-empty, and $sp$ is a square. 
For all $i,j$, we let $\mathfrak{B}_{i,j}$ be the set of words from $\mathfrak{B}$ that can be written $w=pvs$ where $sp$ is a square of period $i$ and $|s|=j$. This implies
$$\mathfrak{B}= \bigcup_{\substack{i\ge1\\j\in\{1,\ldots, 2i-1\}}}\mathfrak{B}_{i,j}\,,$$
which itself implies
\begin{equation}\label{boundonBijCircular}
    |\mathfrak{B}|\le \sum_{i\ge1}\sum_{j=1}^{2i-1}|\mathfrak{B}_{i,j}|\,.
\end{equation}
Fix $i,j$ and without loss of generality, we assume $j\ge i$, the other case being symmetric. 
For $w\in \mathfrak{B}_{i,j}$, we write $w=pvs$ where $|s|=j$ and $sp$ is a square of period $i$.
Let $s=s's''$ such that $s''$ is the suffix of $s$ of length $i$ (which is well-defined since $|s|=j\ge i$). Since $s's''p$ is a square of period $i$, we have $s''=ps'$. That is, for any word $w\in \mathfrak{B}_{i,j}$, the word $w$ is uniquely determined by its prefix of length $|w|-i$. Since this prefix is necessarily square-free, this implies
\begin{equation*}
    |\mathfrak{B}_{i,j}|
    \le |\mathcal{L}_{n-i}(\A, 2)|\le \frac{|\mathcal{L}_{n}(\A, 2)|}{ \beta^i}\,.
\end{equation*} Using this with equations \eqref{mainEqCircular} and \eqref{boundonBijCircular}, we obtain
\begin{align*}
    |\accentset{\circ}{\mathcal{L}}_n(\A, 2)|
    &=|\mathcal{L}_n(\A, 2)|- |\mathfrak{B}|\\
    &\ge|\mathcal{L}_n(\A, 2)|-  \sum_{i\ge1}\sum_{j=1}^{2i-1}|\mathfrak{B}_{i,j}|\\
    &\ge|\mathcal{L}_n(\A, 2)|\left[1 -  \sum_{i\ge1}\sum_{j=1}^{2i-1} \beta^{-i}\right]\\
    &\ge|\mathcal{L}_n(\A, 2)|\left[1 -  \sum_{i\ge1}(2i-1) \beta^{-i}\right]\\
    &\ge|\mathcal{L}_n(\A, 2)|\left[1 - \frac{1+\beta}{(\beta-1)^2}\right]\\
    &=C|\mathcal{L}_n(\A, 2)|\,.
\end{align*}
which concludes our proof.
\end{proof}

We obtain more explicit bounds by combining this result with Corollary \ref{explicitGrowthSquare}.

\begin{corollary}\label{sqfreevssircular}
Let $\A$ be an alphabet with $|\A|\ge5$.
Let 
    \begin{equation*}
        \beta \coloneq \frac{|\A|+\sqrt{|\A|^2-4|\A|}}{2}, \text{ and } C\coloneq1-\frac{1+\beta}{(\beta-1)^2}\,,
    \end{equation*}
    then $C>0$ and 
$$|\accentset{\circ}{\mathcal{L}}_{n}(\A, 2)|\ge C|\mathcal{L}_{n}(\A, 2)|\,.$$

\end{corollary}
\begin{proof}
By Corollary \ref{explicitGrowthSquare}, we have
\[
    |\mathcal L_{n+1}(\A,2)|\ge\beta|\mathcal L_n(\A,2)|\,.
\]

So by definition of $C$, we can apply Theorem \ref{circularsquarefreethm}, and we get 
$|\accentset{\circ}{\mathcal{L}}_n(\A,2)|\ge C\cdot|\mathcal{L}_n(\A, 2)|$, 
with $C=1-\frac{1+\beta}{(\beta-1)^2}$, as desired.
\end{proof}
We give in Table \ref{tableC} of the corresponding values of $C$ for small values of $|\A|$. It is not hard to verify that the $C$ given in  Corollary \ref{sqfreevssircular} behaves like $1-\frac{1}{|\A|}-o\left(\frac{1}{|\A|}\right)$, as $|\A|$ goes to infinity.
\begin{table}
    \center
    \begin{tabular}{c|ccccccccccc}
        $|\A|$ & 5 & 6 & 7 & 8 & 9 & 10 & 11 & 12 & 13 & 14 & 15 \\\hline
         $C$ & 0.32 & 0.58 & 0.70 & 0.76 & 0.81 & 0.84 & 0.86 & 0.87 & 0.89 & 0.90 & 0.91 
    \end{tabular}
    \caption{The values of $C$ given by Corollary \ref{sqfreevssircular}}
    \label{tableC}
\end{table}

We have proven that over any alphabet of size at least $5$ the number of square-free circular words has the same growth rate as the number of square-free words. The question remains open for alphabets of size $3$ and $4$.
It was first proven by Currie \cite{Currie2002Oct} that there are square-free circular ternary words of all lengths other than $5,7,9,10,14,17$. The existence of these six lengths implies that we cannot hope for the existence of a positive constant $C$ such that for all $n$, $ |\accentset{\circ}{\mathcal{L}}_n(\{0,1,2\}, 2)| \ge C |\mathcal{L}_n(\{0,1,2\}, 2)|$. If it is possible to use the same technique for the case of ternary words, it probably requires handling all these cases, which probably entails a computer assisted proof. It might be easier to deal with the alphabet of size $4$, but we do not know how to do it. 

\begin{remark}
The supermultiplicativity (resp. submultiplicativity) of $\left(\mathcal{L}_n(\A, 2)\right)_{n\ge0}$ and the fact that $\accentset{\circ}{\mathcal{L}}_n(\A,2)$ and $\mathcal{L}_n(\A, 2)$ are within a constant multiple of each other imply that the sequence  $\left(|\accentset{\circ}{\mathcal{L}}_n(\A,2)|\right)_{n\ge0}$ is also supermultiplicative (resp. submultiplicative). We wonder whether one could find a more direct proof for any of these properties for $\left(|\accentset{\circ}{\mathcal{L}}_n(\A,2)|\right)_{n\ge0}$ (in this case even submultiplicativity is not as obvious as for the sequence $\left(|\mathcal{L}_n(\A, 2)|\right)_{n\ge0}$). We also wonder the same thing for $\F$-free circular words.
\end{remark}

\paragraph{$\F$-free circular words}
We now state a version of Theorem \ref{circularsquarefreethm} in the general setting of $\F$-free circular words.
For any set $\F\subseteq\A^+$, let $\accentset{\circ}{\mathcal{L}}(\A,\F)$ be the set of words $w$ such that $\accentset{\circ}{w}$ is a $\F$-free circular word over the alphabet $\A$, and for all $n$, $\accentset{\circ}{\mathcal{L}}_n(\A,\F)$ contains those of length $n$. 
\begin{theorem}\label{circularFfreethm}
    Let $\A$ be an alphabet and let $\F\subseteq\A^+$ be a set of factors.
    Suppose there exists $\beta > 1$ such that $|\mathcal{L}_{n+1}(\A, \F)|\ge \beta |\mathcal{L}_n(\A, \F)|$ for all $n\ge 0$, and such that 
    \begin{equation*}
        C\coloneq 1-\sum_{f\in \F} (|f|-1)\beta^{-|f|}>0\,,
    \end{equation*}
    then 
    \begin{equation*}
        |\mathcal{L}_n(\A, \F)|\ge|\accentset{\circ}{\mathcal{L}}_n(\A,\F)|\ge C\cdot|\mathcal{L}_n(\A, \F)|\,.
    \end{equation*}
\end{theorem}

The explicit version of submultiplicativity and supermultiplicativity of $(|\accentset{\circ}{\mathcal{L}}_n(\A,\F)|)_n$ can be written as
\[
    C^2|\accentset{\circ}{\mathcal{L}}_n(\A,\F)||\accentset{\circ}{\mathcal{L}}_m(\A,\F)|\le |\accentset{\circ}{\mathcal{L}}_{n+m}(\A,\F)|\le C^{-2} |\accentset{\circ}{\mathcal{L}}_n(\A,\F)||\accentset{\circ}{\mathcal{L}}_m(\A,\F)|.
\]

The proof of Theorem \ref{circularFfreethm} is an assemblage  of the ideas of the proof of Theorem \ref{thmSupermult} and of the proof of Theorem \ref{circularsquarefreethm}, and it is left to the reader. The remarks of Subsection \ref{Analyticsubsec} also apply to this result, which can be translated into the following corollary.

 \begin{corollary}
    Let $\A$ be an alphabet and let $\F\subseteq\A^+$ be a set of factors.
    Suppose there exists $\beta\ge1$ such that 
    \begin{equation*}
        |\A|>\sum_{f\in \F} \beta^{1-|f|}+\beta\,,
    \end{equation*}
    then the growth rate $\accentset{\circ}{\alpha}(\A,\F)$ of the language $\accentset{\circ}{\mathcal{L}}_n(\A,\F)$ satisfies
    \begin{equation*}
        \accentset{\circ}{\alpha}(\A,\F) = \alpha(\mathcal{L}_n(\A,\F))\ge \beta\,.
    \end{equation*}
\end{corollary}

\paragraph{Circular-free, supermultiplicativity and the Restivo--Salemi property}

Note that $C$ from Theorem \ref{circularFfreethm} (resp. Corollary  \ref{sqfreevssircular}) is the same as in Theorem \ref{thmSupermult} (resp. Corollary \ref{squareFreeSupermult}). However, these constants are not optimal in general, but the optimal constants might be equal as well for many well-behaved languages. Roughly speaking, these two constants give a lower bound on the probability that two valid words can be concatenated without creating a forbidden factor. In Theorem \ref{thmSupermult} the two words are taken independently, while in Theorem \ref{circularFfreethm} these two words are the prefix and suffix of the same word. Our intuition is that for a random $p$-free word the prefix and suffix of length $n$ are almost independent if the word is long enough.
This motivates the following definitions and problem.
For any infinite language $\L(\A,\F)$, we let
\[
\widecheck{C}_t=\frac{|\mathcal{L}_{2t}(\A,\F)|}{|\mathcal{L}_{t}(\A,\F)|\cdot|\mathcal{L}_{t}(\A,\F)|}
\]
and   
\[
\accentset{\circ}{C}_t=\frac{\accentset{\circ}{|\mathcal{L}}_{t}(\A,\F)|}{|\mathcal{L}_{t}(\A,\F)|}\,.
\]
\begin{problem}\label{supermultvscircular}
For which set $\F$ is it true that $\lim\limits_{t\rightarrow\infty} \widecheck{C}_t=\lim\limits_{t\rightarrow\infty} \accentset{\circ}{C}_t$? In particular, when are these limits well-defined?
\end{problem}
We believe that this should be true for many reasonable sets $\F$. In particular, we conjecture that this should be true for $p$-power free words. 
These two quantities are also related to $C_t$, where for all $t$,
\[
C_t=\frac{|\L_n(\A,\F)|}{\alpha(\A,\F)^n}\,.
\]
In particular, we have $\lim\limits_{t\rightarrow\infty} \widecheck{C}_t= \frac{1}{\lim\limits_{t\to\infty}C_t}$, whenever the two limits are well-defined.

It appears that this question might be related to whether or not the language has the Restivo--Salemi property. For any language $\L$, we let $e(\L)$ be the set of words that are infinitely extendable in both directions, that is, $e(\L)=\{w\in\L:\forall n, \exists u,v\in \A^n, uwv\in\L\}$.
Following \cite{Shur2009}, we say that a language has the Restivo--Salemi property,\footnote{The name is motivated by the problems presented by Salemi and Restivo \cite{Restivo1985}.} if for all $u,v\in e(\L)$ there exists $w\in\A^*$ such that $uwv\in e(\L)$. 
In particular, Shur conjectured that any language defined by avoidance of $p$-powers has the Restivo--Salemi property \cite{Shur2009}. 
Recent progress on this conjecture include the cases $|\A|\ge2$ and $p=3$ \cite{Petrova2018Dec} and  $|\A|\ge3$ and $p\ge5$ \cite{Rukavicka2023Dec}. 
The kind of random-like structural properties of a language that imply the  Restivo--Salemi property might be related to the ones that imply the equality in Problem \ref{supermultvscircular}. The technique in this article could also be useful to solve some of the missing cases of Shur's conjecture on the Restivo--Salemi property of $p$-power free languages, but we leave this open as we were not able to make any progress ourselves on this problem.

\section*{Acknowledgments}
The authors sincerely thank the anonymous reviewers for their thoughtful feedback and constructive comments.

\bibliographystyle{unsrt} 
\bibliography{biblio}

\begin{thebibliography}{10}

\bibitem{Rosenfeld2022Apr}
Matthieu Rosenfeld.
\newblock {Finding lower bounds on the growth and entropy of subshifts over countable groups}.
\newblock {\em arXiv}, 2022.

\bibitem{miller}
Joseph Miller.
\newblock Two notes on subshifts.
\newblock {\em Proceedings of the American Mathematical Society}, 140(5):1617--1622, 2012.

\bibitem{Ochem2016Feb}
Pascal Ochem.
\newblock {Doubled Patterns are 3-Avoidable}.
\newblock {\em Electron. J. Combin.}, page P1.19, 2016.

\bibitem{golod}
Evgeny~S. Golod and Igor~R. Shafarevich.
\newblock On the class field tower.
\newblock {\em Izv. Akad. Nauk SSSR Ser. Mat.}, 28:261--272, 1964.

\bibitem{Bell2007Sep}
Jason~P. Bell and Teow~Lim Goh.
\newblock {Exponential lower bounds for the number of words of uniform length avoiding a pattern}.
\newblock {\em Inform. And Comput.}, 205(9):1295--1306, 2007.

\bibitem{Kolpakov2007}
Roman Kolpakov.
\newblock {Efficient lower bounds on the number of repetition-free words.}
\newblock {\em Journal of Integer Sequences}, 10(3), 2007.

\bibitem{Kolpakov2007Dec}
Roman Kolpakov.
\newblock {On the number of repetition-free words}.
\newblock {\em J. Appl. Ind. Math.}, 1(4):453--462, 2007.

\bibitem{Shur2010Jul}
Arseny~M. Shur.
\newblock {Growth rates of complexity of power-free languages}.
\newblock {\em Theoret. Comput. Sci.}, 411(34):3209--3223, 2010.

\bibitem{Shur2014Feb}
Arseny~M. Shur.
\newblock {Growth of Power-Free Languages over Large Alphabets}.
\newblock {\em Theory Comput. Syst.}, 54(2):224--243, 2014.

\bibitem{Shur2012Nov}
Arseny~M. Shur.
\newblock {Growth properties of power-free languages}.
\newblock {\em Computer Science Review}, 6(5):187--208, 2012.

\bibitem{Rosenfeld2024May}
Matthieu Rosenfeld.
\newblock {Ann wins the nonrepetitive game over four letters and the erase-repetition game over six letters}.
\newblock {\em Eur. J. Combin.}, 118:103924, 2024.

\bibitem{Rosenfeld2022Sep}
Matthieu Rosenfeld.
\newblock {Avoiding squares over words with lists of size three amongst four symbols}.
\newblock {\em Math. Comput.}, 91(337):2489--2500, 2022.

\bibitem{Rosenfeld2021Oct}
Matthieu Rosenfeld.
\newblock {Lower-Bounds on the Growth of Power-Free Languages Over Large Alphabets}.
\newblock {\em Theory Comput. Syst.}, 65(7):1110--1116, 2021.

\bibitem{Baker2015Oct}
Simon Baker and Andrei~E. Ghenciu.
\newblock {Dynamical properties of S-gap shifts and other shift spaces}.
\newblock {\em J. Math. Anal. Appl.}, 430(2):633--647, 2015.

\bibitem{Pavlov2022Apr}
Ronnie Pavlov.
\newblock {On subshifts with slow forbidden word growth}.
\newblock {\em Ergod. Theory Dyn. Syst.}, 42(4):1487--1516, 2022.

\bibitem{Thue1}
Axel Thue.
\newblock {\"Uber} die gegenseitige {L}age gleicher {T}eile gewisser {Z}eichenreihen.
\newblock {\em Norske Vid. Selsk. Skr. I. Mat. Nat. Kl. Christiania,}, 10:1--67, 1912.

\bibitem{Shur2010Oct}
Arseny~M. Shur.
\newblock {On Ternary Square-Free Circular Words}.
\newblock {\em Electron. J. Combin.}, 17, 2010.

\bibitem{Shur2009}
Arseny~M. Shur.
\newblock {Two-Sided Bounds for the Growth Rates of Power-Free Languages}.
\newblock In {\em {Developments in Language Theory}}, pages 466--477. 2009.

\bibitem{Currie2020May}
James~D. Currie and Jesse~T. Johnson.
\newblock {There are level ternary circular square-free words of length $n$ for $n\ne 5,7,9,10,14,17.$}.
\newblock {\em arXiv}, 2020.

\bibitem{Currie2003Jun}
James~D. Currie and D.~Sean Fitzpatrick.
\newblock {Circular Words Avoiding Patterns}.
\newblock In {\em {Developments in Language Theory}}, pages 319--325. 2003.

\bibitem{Currie2021Jan}
James~D. Currie and Jesse~T. Johnson.
\newblock {Characterization of the lengths of binary circular words containing no squares other than 00, 11, and 0101}.
\newblock {\em Theoret. Comput. Sci.}, 850:30--39, 2021.

\bibitem{Currie2002Oct}
James~D. Currie.
\newblock {There Are Ternary Circular Square-Free Words of Length $n$ for $n\ge 18$}.
\newblock {\em Electron. J. Combin.}, 9(1), 2002.

\bibitem{Restivo1985}
Antonio Restivo and Sergio Salemi.
\newblock {Some Decision Results on Nonrepetitive Words}.
\newblock In {\em {Combinatorial Algorithms on Words}}, pages 289--295. 1985.

\bibitem{Petrova2018Dec}
Elena~A. Petrova and Arseny~M. Shur.
\newblock {Transition Property for Cube-Free Words}.
\newblock {\em Theory Comput. Syst.}, 65(3):479--496, 2021.

\bibitem{Rukavicka2023Dec}
Josef Rukavicka.
\newblock {Restivo Salemi property for $\alpha$-power free languages with $\alpha\geq 5$ and $k\geq 3$ letters}.
\newblock {\em arXiv}, 2023.

\end{thebibliography}

\end{document}